\documentclass{amsart}

\usepackage{amsthm,amsmath,amssymb,amscd,amsfonts,latexsym}

\theoremstyle{plain}
\newtheorem{theorem}{Theorem}[section]

\newtheorem{def-thm}[theorem]{Definition-Theorem}
\newtheorem{lemma}[theorem]{Lemma}

\theoremstyle{definition}

\newtheorem{remark}[theorem]{Remark}
\newtheorem{example}[theorem]{Example}

\newtheorem*{acknowledgement}{Acknowledgement}

\newcommand{\PP}{\mathbb{P}}

\newcommand{\OO}{{\mathcal O}}

\allowdisplaybreaks

\begin{document}

\title[Metrics of Positive Holomorphic Sectional Curvature]{Hermitian Metrics of Positive Holomorphic Sectional Curvature on Fibrations}

\begin{abstract}
The main result of this note essentially is that if the base and fibers of a compact fibration carry Hermitian metrics of positive holomorphic sectional curvature, then so does the total space of the fibration. The proof is based on the use of a warped product metric as in the work by Cheung in case of negative holomorphic sectional curvature, but differs in certain key aspects, e.g., in that it does not use the subadditivity property for holomorphic sectional curvature due to Grauert-Reckziegel and Wu.\end{abstract}

\author{Ananya Chaturvedi}

\address{Department of Mathematics\\Tata Institute of Fundamental Research\\Dr Homi Bhabha Road, Navy Nagar, Colaba, Mumbai, Maharashtra 400005\\India}
\email{ananya@math.tifr.res.in}

\author{Gordon Heier}
\address{Department of Mathematics\\University of Houston\\4800 Calhoun Road, Houston, TX 77204\\USA}
\email{heier@math.uh.edu}



\maketitle

\section{Introduction}
There has recently been strong renewed interest in the interplay of (semi-)positive or (semi-)negative holomorphic sectional curvature and the structure of Hermitian or K\"ahler manifolds in complex differential geometry. Under assumptions of the existence of a metric of (semi-)definite holomorphic sectional curvature, the works \cite{heier_lu_wong_mrl}, \cite{HLW_JDG}, \cite{Wu_Yau_Invent}, \cite{Tosatti_Yang}, \cite{Wu_Yau_CAG}, \cite {Diverio_Trapani}, \cite{HLWZ} (for (semi-)negative holomorphic sectional curvature) and \cite{Tsukamoto_1957}, \cite{heier_wong_cag},  \cite{heier_wong_arxiv}, \cite{Yang} (for (semi-)positive holomorphic sectional curvature) have established interesting consequences for the geometric structure of such manifolds, in particular concerning the canonical line bundle. \par

Conversely, it is an interesting problem to prove the existence of metrics of (semi-)definite holomorphic sectional curvature under the assumption of a certain geometric structure. In this direction, Grauert-Reckziegel \cite{Grauert_Reckziegel} proved the existence of a neighborhood of a fiber of a holomorphic family of compact Riemann surfaces of genus at least $2$ over a Riemann surface such that the neighborhood has negative holomorphic sectional curvature. Cowen \cite{Cowen} extended this work to the case where the fibers are Hermitian compact complex manifolds of arbitrary dimension whose holomorphic sectional curvature is negative. Subsequently, Cheung \cite[Theorem 1]{Cheung} established the following (optimal) result. Let $X,Y$ be complex manifolds. Let $\pi:X\to Y$ be a surjective submersion, which we subsequently will refer to as a {\it fibration}. Assume that $X$ (and thus $Y$) is compact, making $\pi$ into a {\it compact fibration}. Furthermore, assume that $Y$ possesses a Hermitian metric of negative holomorphic sectional curvature and that there exists a smooth family of Hermitian metrics on the fibers such that each metric is of negative holomorphic sectional curvature. Then there exists a Hermitian metric on $X$ which is of negative holomorphic sectional curvature everywhere. \par

On the side of positive curvature, Hitchin \cite{Hitchin} proved that Hirzebruch surfaces, which are $\PP^1$-bundles over $\PP^1$, carry a (Hodge) metric of positive holomorphic sectional curvature. The explicit values and pinching constants for Hitchin's metrics were computed in \cite{ACH}, which was further generalized in \cite{yang_zheng}. The fact that Hirzebruch surfaces are isomorphic to the projectivization of vector bundles $\PP(\OO_{\PP^1}(a) \oplus  \OO_{\PP^1})$, $a\in \{0,1,2,3,\ldots\}$, over $\PP^1$, motivated the statement and proof of the main theorem in \cite{AHZ}, establishing the existence of a K\"ahler metric of positive holomorphic sectional curvature on an arbitrary projectivized holomorphic vector bundle over a compact K\"ahler base manifold of positive holomorphic sectional curvature. The main result of this note is the following full-fledged positive curvature analog for Cheung's theorem.

\begin{theorem}\label{mthm}
Let $\pi: X \rightarrow Y$ be a compact fibration. Assume that $Y$ has a Hermitian metric of positive holomorphic sectional curvature and that there exists a smooth family of Hermitian metrics on the fibers which all have positive holomorphic sectional curvature. Then there exists a Hermitian metric on $X$ with positive holomorphic sectional curvature.
\end{theorem}

The main strategy in Cheung's proof is to construct a {\it warped product metric} from the ``vertical" metrics and a large real constant multiple of the pull-back of the metric from the base. The idea of using warped product metrics to obtain negative curvature metrics seems to go back to the work of Bishop-O'Neill in \cite{Bishop_O'Neill}, and it has become an important tool in differential geometry and theoretical physics, especially in general relativity in the study of solutions of Einstein's equation. Our goal was to transpose Cheung's proof to the positive case and in doing so, find replacements for the key estimates in Cheung's proof that rely on (a) the curvature decreasing property of Hermitian subbundles and (b) the subadditivity property of holomorphic sectional curvature due to \cite[Aussage 1]{Grauert_Reckziegel} and \cite[Theorem 1]{Wu}). Both of these properties appear to ``go in the wrong direction" in the positive case, which adds an additional layer of complication to the present problem and generally to many problems in complex differential geometry in positive curvature.

\begin{remark}
It should be mentioned that we are strictly working in the Hermitian setting in this paper and that the constructed metric on the total space $X$ in general will not be K\"ahler, even if all the given metrics are K\"ahler. While most of the interesting consequences for the geometric structure mentioned in the very first paragraph require a K\"ahler assumption, it was shown in \cite[Theorem 1.2]{Yang} that a Hermitian manifold of positive holomorphic sectional curvature has negative Kodaira dimension. Therefore, we know in the situation of Theorem \ref{mthm} that the total space $X$ has negative Kodaira dimension as well. However, there is no need to use Theorem 1.1 to conclude this: it is in fact an immediate consequence of the Easy Addition Formula for Kodaira dimensions. Note that our theorem also bears resemblance to the result of Graber-Harris-Starr \cite{GHS} yielding, roughly speaking, the rational connectedness of the total space due to the rational connectedness of the base and fibers. Recall also that it was proven in \cite{heier_wong_arxiv} that a projective K\"ahler manifold of positive holomorphic sectional curvature is rationally connected.
\end{remark}

We would like to remark here that unlike in the definite curvature cases, the warping technique does not seem to work for the semi-definite curvature analogs of Theorem \ref{mthm}. In particular, if we have positive holomorphic sectional curvature on the base and semi-positive holomorphic sectional curvature along each fiber, then the following Example \ref{ex_pos_sempos} is a counterexample to the analogous statement of Lemma \ref{lem1} in this case. The warp factor $\lambda$ in Example \ref{ex_pos_sempos} may be non-constant and is merely assumed to depend smoothly on the points of the base space $Y$, whereby we show that a clever choice of a non-constant warping factor cannot help to accomplish the goal. For complete details of the computations of the curvatures in Example \ref{ex_pos_sempos}, we refer the reader to the first author's thesis \cite{thesis}, which also contains further counterexamples along these lines. Naturally, it remains an open question if the semi-definite cases can be handled by methods other than warped products.
\begin{example}\label{ex_pos_sempos}
Let $X = D_1 \times D_2$ be the bi-disk, $Y = D_2$, and $(z_1,z_2)$ be the coordinate system in $X$. The holomorphic map $\pi: X \rightarrow Y$ is given by the projection onto the second coordinate, i.e., $(z_1,z_2) \mapsto z_2$. Clearly, $\pi$ is of maximal rank everywhere. The Hermitian metric $\omega_Y =\ 1/\big(1+z_2\bar z_2\big)\ dz_2 \otimes d\bar z_2$ on $Y$ has positive holomorphic sectional curvature everywhere on $Y$, and the following tensor on $X$
\[\Phi =\ \frac{e^{2z_2\bar z_2}}{1 + (z_1\bar z_1)^2 e^{4z_2\bar z_2}}\ dz_1 \otimes d\bar z_1\]
yields a Hermitian metric of semi-positive holomorphic sectional curvature when restricted to any of the fibers $\pi^{-1}(z_2), \, z_2 \in Y$. Clearly, $\Phi$ varies smoothly with respect to the base points in $Y$. Also, $\Phi + \mu \pi^*\omega_Y$ is a Hermitian metric on $X$ for any $\mu > 0$. However, for any smooth real-valued positive function $\lambda(z_2)$ on $Y$, the Hermitian metric
\begin{equation*}
\begin{aligned}
G =\ & \Phi + \lambda \pi^*\omega_Y \\
  =\ & \frac{e^{2z_2\bar z_2}}{1 + (z_1\bar z_1)^2 e^{4z_2\bar z_2}}\ dz_1 \otimes d\bar z_1 + \frac{\lambda}{1+z_2\bar z_2}\ dz_2 \otimes d\bar z_2
\end{aligned}
\end{equation*}
does not have semi-positive holomorphic sectional curvature everywhere on $X$.\qed
\end{example}

\par This paper is organized as follows. In Section \ref{sec_def}, we will recall basic definitions. In Section \ref{lemmas}, we will state and prove two key lemmas. In Section \ref{proof_mthm}, we will prove Theorem \ref{mthm}. Finally, in Section 5, we explain why our method fails in general if $\pi$ is not a submersion.

\begin{acknowledgement}
This work was part of the first author's Ph.D. thesis written under the direction of the second author at the University of Houston. The first author would like to thank TIFR Center for Applicable Mathematics, Bengaluru and TIFR, Mumbai for support during the time this paper was finalized.
\end{acknowledgement}

\section{Basic definitions}\label{sec_def}
Let $M$ be an $n$-dimensional manifold with local coordinates $z_1,\ldots,z_n$. Let
\begin{equation*}
\sum_{i,j=1}^n g_{i\bar j} dz_i\otimes d\bar z_j
\end{equation*}
be a Hermitian metric on $M$. Under the usual abuse of terminology, we will alternatively refer to the associated (1,1)-form $\omega=\frac{\sqrt{-1}}{2}\sum_{i,j=1}^n g_{i\bar j} dz_i\wedge d\bar z_j$ as the metric on $M$.
\par The components $R_{i\bar j k \bar l}$ of the curvature tensor $R$ associated with the metric connection are locally given by the formula
\begin{equation}\label{curv_formula}
R_{i\bar j k \bar l}=-\frac{\partial^2 g_{i\bar j}}{\partial z_k\partial \bar z _l}+\sum_{p,q=1}^n g^{p\bar q}\frac{\partial g_{i\bar p}}{\partial z_k}\frac{\partial g_{q\bar j}}{\partial \bar z_l}.
\end{equation}\par

If $\xi=\sum_{i=1}^n\xi_i \frac{\partial }{\partial z_i}$ is a non-zero complex tangent vector at $p\in M$, then the {\it holomorphic sectional curvature} $K(\xi)$ is given by
\begin{equation}\label{hol_sect_curv_def}
K(\xi)=\left( 2 \sum_{i,j,k,l=1}^n R_{i\bar j k \bar l}(p)\xi_i\bar\xi_j\xi_k\bar\xi_l\right) \Bigg/ \left(\sum_{i,j,k,l=1}^n g_{i\bar j}g_{k\bar l} \xi_i\bar\xi_j\xi_k\bar\xi_l\right).
\end{equation}
Note that the holomorphic sectional curvature of $\xi$ is clearly invariant under multiplication of $\xi$ with a real non-zero scalar, and it thus suffices to consider unit vectors, for which the value of the denominator is $1$.

\section{Two lemmas}\label{lemmas}
In this section, we state and prove two lemmas that pave the way to the proof of Theorem \ref{mthm} in the next section.
\begin{lemma} \label{lem1}
Let $M$ be an $n$-dimensional Hermitian manifold, and $G$ be a Hermitian metric on $M$. Let $R_{i \bar j k \bar l}$ be the components of the curvature tensor with respect to $G$ for $i,j,k,l = 1, \ldots, n$. Suppose the following is true at a point $p \in M$ for some positive constants $K_0, K_1, K_2$, and a natural number $s < n$:
\begin{enumerate}
  \item \[\sum_{i,j,k,l=1}^s R_{i \bar j k \bar l}(p) \, \xi_i\bar\xi_j\xi_k\bar\xi_l \geq K_0 \sum_{i,j=1}^s \xi_i\bar\xi_i\xi_j\bar\xi_j\]
        for all $\xi_i \in \mathbb{C}, \quad i=1,2,\ldots,s$.
  \item \[|R_{i \bar j k \bar l}(p)| < K_1\]
        whenever $\min (i,j,k,l) \leq s$ and $\max (i,j,k,l) > s$.
  \item \[\sum_{\alpha,\beta,\gamma,\delta = s+1}^n R_{\alpha\bar\beta\gamma\bar\delta}(p) \, \xi_\alpha\bar\xi_\beta\xi_\gamma\bar\xi_\delta \geq K_2 \sum_{\alpha,\beta = s+1}^n \xi_\alpha\bar\xi_\alpha\xi_\beta\bar\xi_\beta\]
        for all $\xi_\alpha \in \mathbb{C}, \quad \alpha = s+1,s+2,\ldots,n$.
\end{enumerate}
Then there exists a positive constant $\mathcal{K}$ depending only on $K_0/K_1$ such that if $K_2/K_1 \geq \mathcal{K}$, then $G$ has positive holomorphic sectional curvature at the point $p$.
\end{lemma}

\begin{proof}
Suppose the Hermitian metric $G$ on $M$ is given locally by
\begin{equation*}
G = \sum_{i,j=1}^n g_{i\bar j} \, dz_i \otimes d\bar z_j.
\end{equation*}
Since we are only interested in the sign of the holomorphic sectional curvature, it suffices to check the numerator of \eqref{hol_sect_curv_def} for positive sign at $p$ in the direction of $\xi = (\xi_1,\ldots,\xi_n)$ with respect to the Hermitian metric $G$ which is given as follows:
\begin{equation}
\begin{aligned}\label{curv_mod}
  & \sum_{i,j,k,l=1}^n R_{i\bar jk\bar l}(p) \, \xi_i\bar\xi_j\xi_k\bar\xi_l \\
      = & \, \sum_{i,j,k,l=1}^s R_{i \bar jk\bar l}(p) \, \xi_i\bar \xi_j\xi_k\bar\xi_l + \sum_{\substack{i,j,k,l=1 \\ \min(i,j,k,l) \leq s, \\ \max(i,j,k,l) > s}}^n R_{i\bar jk\bar l}(p) \, \xi_i\bar\xi_j\xi_k\bar\xi_l \\
        & + \sum_{\alpha,\beta,\gamma,\delta = s+1}^n R_{\alpha\bar\beta\gamma\bar\delta}(p) \, \xi_\alpha\bar\xi_\beta\xi_\gamma\bar\xi_\delta \\
   \geq & \, \sum_{i,j,k,l=1}^s R_{i\bar jk\bar l}(p) \, \xi_i\bar\xi_j\xi_k\bar\xi_l - \Bigg|\sum_{\substack{i,j,k,l=1 \\ \min(i,j,k,l) \leq s, \\ \max(i,j,k,l) > s}}^n R_{i\bar jk\bar l}(p) \, \xi_i\bar\xi_j\xi_k\bar\xi_l\Bigg| \\
        & + \sum_{\alpha,\beta,\gamma,\delta = s+1}^n R_{\alpha\bar\beta\gamma\bar\delta}(p) \, \xi_\alpha\bar\xi_\beta\xi_\gamma\bar\xi_\delta \\
   \geq & \, \sum_{i,j,k,l=1}^s R_{i\bar jk\bar l}(p) \,\xi_i\bar\xi_j\xi_k\bar\xi_l - \sum_{\substack{i,j,k,l=1 \\ \min(i,j,k,l) \leq s, \\ \max(i,j,k,l) > s}}^n \big|R_{i\bar jk\bar l}(p)\big| \, |\xi_i| |\xi_j| |\xi_k| |\xi_l| \\
        & + \sum_{\alpha,\beta,\gamma,\delta = s+1}^n R_{\alpha\bar\beta\gamma\bar\delta}(p) \, \xi_\alpha\bar\xi_\beta\xi_\gamma\bar\xi_\delta.
\end{aligned}
\end{equation}
There are three possible options in the second summation for $\min(i,j,k,l) \leq s$ and $\max(i,j,k,l) > s$:
\begin{enumerate}
  \item Only one of the four indices is less than or equal to $s$, corresponding to \textit{four} types of summands.
  \item Exactly two out of the four indices are less than or equal to $s$, corresponding to \textit{six} types of summands.
  \item Three out of the four indices are less than or equal to $s$, corresponding to \textit{four} types of summands.
\end{enumerate}
Substituting the above facts along with the hypotheses of the lemma in \eqref{curv_mod}, we obtain
\begin{equation}
\begin{aligned}\label{Cp_curv}
  \sum_{i,j,k,l=1}^n R_{i\bar jk\bar l}(p) \, \xi_i\bar\xi_j\xi_k\bar\xi_l \geq & \, K_0 \sum_{i,j=1}^s \xi_i\bar\xi_i\xi_j\bar\xi_j - 4K_1 \sum_{\alpha, \beta, \gamma = s+1}^n \sum_{i=1}^s |\xi_i||\xi_\alpha||\xi_\beta||\xi_\gamma| \\
           & - 6K_1  \sum_{\alpha, \beta = s+1}^n \sum_{i, j = 1}^s |\xi_i| |\xi_j| |\xi_\alpha| |\xi_\beta| \\
           & - 4K_1 \sum_{\alpha = s+1}^n \sum_{i, j, k = 1}^s |\xi_i| |\xi_j| |\xi_k| |\xi_\alpha| + K_2 \sum_{\alpha, \beta = s+1}^n \xi_\alpha\bar\xi_\alpha\xi_\beta\bar\xi_\beta
\end{aligned}
\end{equation}
for any choice of $\xi = (\xi_1,\ldots,\xi_n) \in \mathbb{C}^n$. For any choice of positive numbers $a,b,c,d$, we have:
\begin{equation}\label{prod_sum_ineq}
\begin{aligned}
  |\xi_i||\xi_\alpha||\xi_\beta||\xi_\gamma| \leq & \, a^2 |\xi_i|^4 + \frac{1}{a^2} |\xi_\alpha|^4 + |\xi_\beta|^2|\xi_\gamma|^2, \\
  |\xi_i||\xi_j||\xi_\alpha||\xi_\beta| \leq & \, b^2 |\xi_i|^2|\xi_j|^2 + \frac{1}{b^2} |\xi_\alpha|^2|\xi_\beta|^2, \\
  |\xi_i||\xi_j||\xi_k||\xi_\alpha| \leq & \, c^2 |\xi_i|^2|\xi_j|^2 + \frac{d^2}{c^2} |\xi_k|^4 + \frac{1}{c^2d^2} |\xi_\alpha|^4.
\end{aligned}
\end{equation}\par

Substituting the inequalities from \eqref{prod_sum_ineq} into \eqref{Cp_curv} and using the trivial inequalities $\sum_{i=1}^s |\xi_i|^4 \leq \sum_{i,j=1}^s |\xi_i|^2|\xi_j|^2$ and $\sum_{\alpha=s+1}^n |\xi_\alpha|^4 \leq \sum_{\alpha, \beta = s+1}^n |\xi_\alpha|^2|\xi_\beta|^2$, we obtain
\begin{equation*}
\begin{aligned}
  \sum_{i,j,k,l=1}^n R_{i\bar jk\bar l}(p) \, \xi_i\bar\xi_j\xi_k\bar\xi_l \geq\ & \, K_0 \sum_{i,j=1}^s |\xi_i|^2|\xi_j|^2 - K_1\Bigg(\bigg(4a^2(n-s)^3 + 6b^2(n-s)^2 \\
           & + 4c^2s(n-s) + \frac{4d^2}{c^2}s^2(n-s)\bigg) \sum_{i, j = 1}^s |\xi_i|^2|\xi_j|^2 \\
           & + \bigg(\frac{4}{a^2} s(n-s)^2 + 4s(n-s) + \frac{6}{b^2}s^2 \\
           &+ \frac{4}{c^2d^2}s^3\bigg) \sum_{\alpha,\beta=s+1}^n |\xi_\alpha|^2|\xi_\beta|^2\Bigg) + K_2 \sum_{\alpha, \beta = s+1}^n |\xi_\alpha|^2|\xi_\beta|^2.
\end{aligned}
\end{equation*}\par
We may choose $a,b,c,d$ such that
\[4a^2(n-s)^3 + 6b^2(n-s)^2 + 4c^2s(n-s) + \frac{4d^2}{c^2}s^2(n-s) \leq \frac{1}{2} \frac{K_0}{K_1}.\]
Let $\mathcal{K} = \frac{4}{a^2} s(n-s)^2 + 4s(n-s) + \frac{6}{b^2}s^2 + \frac{4}{c^2d^2}s^3$. Note that since the choice of $a,b,c,d$ is based on $K_0/K_1$, therefore $\mathcal{K}$ too depends only on $K_0/K_1$. Then for such a choice of $a, b, c$ and $d$,

\begin{equation*}
\begin{aligned}
  \sum_{i,j,k,l=1}^n R_{i\bar jk\bar l}(p) \, \xi_i\bar\xi_j\xi_k\bar\xi_l \geq & \, K_0 \sum_{i,j=1}^s |\xi_i|^2|\xi_j|^2 - K_1\bigg(\frac{K_0}{2K_1} \sum_{i, j = 1}^s |\xi_i|^2|\xi_j|^2 \\
           & + \mathcal{K} \sum_{\alpha,\beta=s+1}^n |\xi_\alpha|^2|\xi_\beta|^2\bigg) + K_2 \sum_{\alpha, \beta = s+1}^n |\xi_\alpha|^2|\xi_\beta|^2 \\
         = & \frac{K_0}{2} \sum_{i,j=1}^s |\xi_i|^2|\xi_j|^2 + \big(K_2 - K_1\mathcal{K}\big) \sum_{\alpha, \beta = s+1}^n |\xi_\alpha|^2|\xi_\beta|^2.
\end{aligned}
\end{equation*}\par
If $K_2-K_1\mathcal{K} \geq 0$, i.e., $\mathcal{K} \leq K_2/K_1$, then clearly $2\sum_{i,j,k,l=1}^n R_{i\bar jk\bar l}(p) \, \xi_i\bar\xi_j\xi_k\bar\xi_l$ is positive, which is the numerator of the holomorphic sectional curvature in the direction of a tangent vector $(\xi_1,\ldots,\xi_n)$ as given in \eqref{hol_sect_curv_def}. Since the denominator of \eqref{hol_sect_curv_def} is always positive, we conclude that the holomorphic sectional curvature at $p$ with respect to $G$ is positive in the direction of $(\xi_1,\ldots,\xi_n)$, if the above condition is satisfied \big(i.e., $\mathcal{K} \leq K_2/K_1$\big).
\end{proof}

\begin{lemma} \label{lem2}
Let $M$ be an $n$-dimensional complex manifold with two Hermitian metrics $G$ and $H$ defined on it. Suppose that the metric $H$ has positive holomorphic sectional curvature at a point $p \in M$. Then $G + \lambda H$ also has positive holomorphic sectional curvature at $p$ for $\lambda$ large enough.
\end{lemma}

\begin{proof}
For the given point $p \in M$ and a unit tangent vector $t$ at $p$, we choose local coordinates $(z_1,\ldots,z_n)$ at $p$ which satisfy the conditions in \cite[Lemma 3]{Wu} with respect to $H$, i.e.,

\begin{enumerate}
  \item $z_1(p) = \ldots = z_n(p) = 0.$
  \item If $H = \sum_{i,j=1}^n h_{i\bar j}\ dz_i \otimes d\bar z_j$, then $h_{i\bar j}(p) = \delta_{i\bar j}$ and
        \[\frac{\partial h_{i\bar j}}{\partial z_n}(p) = \frac{\partial h_{i\bar j}}{\partial\bar z_n}(p) = 0,\]
        for $1 \leq i, \, j \leq n$.
  \item $t = \frac{\partial}{\partial z_n}(p).$
\end{enumerate}

\par Consider the 1-dimensional complex submanifold $M^{\prime} = \{z_1 = \ldots = z_{n-1} = 0\}$, which is tangent to $t$. Using  \cite[Lemma 4]{Wu}, the Gaussian curvature of $M^{\prime}$ at $p$ with respect to the induced metric $H^{\prime} \big(= H|_{M^{\prime}} = h_{n\bar n}\ dz_n \otimes d\bar z_n\big)$ equals the holomorphic sectional curvature at $p$ with respect to $H$ in the direction of $t$, denoted by $K(H,t)(p)$. Moreover, if $G = \sum_{i,j=1}^n g_{i\bar j} \, dz_i \otimes d\bar z_j$, then the induced metric of $G$ on $M^{\prime}$ is given by $G^{\prime} = G|_{M^{\prime}} = g_{n\bar n} dz_n \otimes d\bar z_n$. Let us denote $g = g_{n\bar n}$, $h = h_{n\bar n}$, and $z = z_n$. Then, $G^{\prime} + \lambda H^{\prime} = (g + \lambda h) \, dz \otimes d\bar z$ is the induced metric of $G + \lambda H$ on $M^{\prime}$. The holomorphic sectional curvature at $p$ with respect to $G^{\prime} + \lambda H^{\prime}$ is given by

\begin{equation} \label{HSC_lambda_depend_comp}
\begin{aligned}
   &   K\big(G^{\prime} + \lambda H^{\prime}\big)(p) \\
        =\ & \frac{2}{\big(g(p)+\lambda h(p)\big)^2} \Bigg(-\frac{\partial^2 \big(g+\lambda h\big)}{\partial z \partial \bar z}(p) \\
           & + \frac{1}{\big(g(p)+\lambda h(p)\big)} \frac{\partial \big(g+\lambda h\big)}{\partial z}(p) \frac{\partial \big(g+\lambda h\big)}{\partial\bar z}(p)\Bigg)\\
        =\ & \frac{2}{\big(g(p)+\lambda h(p)\big)^3} \Bigg(-\big(g(p)+\lambda h(p)\big)\bigg(\frac{\partial^2 g}{\partial z \partial\bar z}(p) +\lambda \frac{\partial^2 h}{\partial z \partial\bar z}(p)\bigg) \\
           & + \bigg(\frac{\partial g}{\partial z}(p) + \lambda\frac{\partial h}{\partial z}(p)\bigg)\bigg(\frac{\partial g}{\partial\bar z}(p) + \lambda\frac{\partial h}{\partial\bar z}(p)\bigg)\Bigg)\\
        =\ & \frac{1}{\big(g(p)+\lambda h(p)\big)^3} \Bigg(2g(p)\bigg(-\frac{\partial^2 g}{\partial z \partial\bar z}(p) + \frac{1}{g(p)} \frac{\partial g}{\partial z}(p) \frac{\partial g}{\partial\bar z}(p)\bigg) \\
           & + 2\lambda^2h(p)\bigg(-\frac{\partial^2 h}{\partial z \partial\bar z}(p) + \frac{1}{h(p)} \frac{\partial h}{\partial z}(p) \frac{\partial h}{\partial\bar z}(p)\bigg) \\
           & + 2\lambda \bigg(-h(p)\frac{\partial^2 g}{\partial z \partial\bar z}(p) - g(p)\frac{\partial^2 h}{\partial z \partial\bar z}(p) + \frac{\partial g}{\partial z}(p) \frac{\partial h}{\partial\bar z}(p) + \frac{\partial h}{\partial z}(p) \frac{\partial g}{\partial\bar z}(p)\bigg)\Bigg).
\end{aligned}
\end{equation}
Therefore,
\begin{equation} \label{HSC_lambda_depend}
\begin{aligned}
  & K\big(G^{\prime} + \lambda H^{\prime}\big)(p) \\
     =\ & \frac{1}{\big(g(p)+\lambda h(p)\big)^3} \Bigg(g(p)^3 K(G^{\prime})(p) + \lambda^2 h(p)^3 K(H^{\prime})(p) \\
        & + 2\lambda \bigg(-h(p)\frac{\partial^2 g}{\partial z \partial\bar z}(p) - g(p)\frac{\partial^2 h}{\partial z \partial\bar z}(p) + \frac{\partial g}{\partial z}(p) \frac{\partial h}{\partial\bar z}(p) + \frac{\partial h}{\partial z}(p) \frac{\partial g}{\partial\bar z}(p)\bigg)\Bigg),
\end{aligned}
\end{equation}
where $K(G^{\prime})(p)$ and $K(H^{\prime})(p)$ are the holomorphic sectional curvatures at $p$ with respect to $G^{\prime}$ and $H^{\prime}$, respectively. The choice of $M^{\prime}$ was such that $K(H^{\prime})(p) = K(H,t)(p)$. Moreover, the decreasing property of holomorphic sectional curvature on submanifolds implies that $K\big(G + \lambda H, t\big)(p) \geq K\big(G^{\prime} + \lambda H^{\prime}\big)(p)$, where $K(G + \lambda H, t)(p)$ denotes the holomorphic sectional curvature at $p$ with respect to $G + \lambda H$ in the direction of $t$. Therefore, \eqref{HSC_lambda_depend} implies that

\begin{equation}\label{HSC_curv_dec}
\begin{aligned}
     &   K\big(G + \lambda H, t\big)(p) \\
      \geq\ &  K\big(G^{\prime} + \lambda H^{\prime}\big)(p) \\
         =\ & \frac{1}{\big(g(p)+\lambda h(p)\big)^3} \Bigg(g(p)^3 K(G^{\prime})(p) + \lambda^2 h(p)^3 K(H,t)(p) + 2\lambda \bigg(-h(p)\frac{\partial^2 g}{\partial z \partial\bar z}(p) \\
            & - g(p)\frac{\partial^2 h}{\partial z \partial\bar z}(p) + \frac{\partial g}{\partial z}(p) \frac{\partial h}{\partial\bar z}(p) + \frac{\partial h}{\partial z}(p) \frac{\partial g}{\partial\bar z}(p)\bigg)\Bigg).
\end{aligned}
\end{equation}

\par If $\lambda$ is large enough, then the sign of the expression on the right hand side is determined by the sign of $K(H,t)(p)$, which is positive by assumption. Hence, the holomorphic sectional curvature at $p$ with respect to $G + \lambda H$ is positive in the direction of $t$ for a sufficiently large value of $\lambda$, say $\lambda_t$, i.e.,
\[K(G + \lambda_t H, t)(p) > 0.\]\par

It is clear from \eqref{HSC_curv_dec} that for any $\lambda > \lambda_t, \; K(G + \lambda H,t)(p)$ is still positive. Finally, we seek a $\lambda$ which works for all tangent vectors at $p$. Any tangent vector at $p$ is a scalar multiple of some unit tangent vector at $p$ (with respect to $H$). Therefore, it is enough to consider $\lambda_t$ for $t \in S^1\big(T_pM\big) = \{t \in T_pM \, : \; H(t,t) = 1\}$. Since $S^1\big(T_pM\big)$ is compact, we conclude that there exists a $\lambda$ which works for all $t$, i.e., $K(G + \lambda H, t)(p)$ is positive for any choice of tangent vector $t$ at $p$. Thus, the holomorphic sectional curvature at $p$ with respect to $G + \lambda H$ is positive in all the directions.
\end{proof}

\begin{remark} \label{Tensor_sum_order}
\par The right hand side of the inequality \eqref{HSC_curv_dec} is $O(\lambda^{-1})$. Thus, if we write $R_{i\bar jk\bar l}$ for the coefficients of the curvature tensor of the metric $G + \lambda H$, the formula \eqref{hol_sect_curv_def} of holomorphic sectional curvature implies that
\begin{equation}\label{lambda_exp} \frac{2\sum_{i,j,k,l=1}^n R_{i\bar jk\bar l}(p)\ \xi_i\bar\xi_j\xi_k\bar\xi_l}{\sum_{i,j,k,l=1}^n \big(g_{i\bar j}(p) + \lambda h_{i\bar j}(p)\big)\big(g_{k\bar l}(p) + \lambda h_{k\bar l}(p)\big)\ \xi_i\bar\xi_j\xi_k\bar\xi_l} \geq O(\lambda^{-1}),\end{equation}
which further implies that
\begin{equation} \label{key_remark}
\sum_{i,j,k,l=1}^n R_{i\bar jk\bar l}(p) \xi_i\bar\xi_j\xi_k\bar\xi_l \geq O(\lambda) \sum_{i,j=1}^n \xi_i\bar\xi_i\xi_j\bar\xi_j.
\end{equation}
\end{remark}

\section{Proof of Theorem \ref{mthm}}\label{proof_mthm}
Suppose $\{G_t\}$ is a smooth family of Hermitian metrics with positive holomorphic sectional curvature on each fiber, and $\varphi_t$ is the Hermitian form associated to the metric $G_t$. Fix a Hermitian metric $\widetilde G$ on $X$. For two vector fields $Z_1$ and $Z_2$ on $X$ of type $(1,0)$ and $(0,1)$, respectively, we define a (1,1)-form  $\widetilde\Phi$ at a point $p \in \pi^{-1}(t)$ as follows:
\[\widetilde\Phi(Z_1,Z_2)(p) \equiv \varphi_t \big(\text{proj}_{\widetilde G} Z_1(p), \text{proj}_{\widetilde G} Z_2(p)\big),\]
where proj$_ {\widetilde G}$ is the projection onto the fiber direction with respect to the metric $\widetilde G $. Then clearly $\widetilde\Phi$ is a $C^\infty$, semi-positive definite Hermitian (1,1)-form defined on $X$, and $\widetilde\Phi$ restricted to each fiber is equal to $\varphi_t$, which is positive definite.
\par Suppose $\omega_Y$ is the associated (1,1)-form of the Hermitian metric on $Y$ with positive holomorphic sectional curvature. Then for any positive value of $\mu$, $\widetilde\Phi + \mu \pi^*\omega_Y$ is a positive definite Hermitian (1,1)-form defined on $X$. We fix one such $\mu_0$, and consider the Hermitian metric on $X$ with the associated (1,1)-form given by $\Phi = \widetilde\Phi + \mu_0 \pi^*\omega_Y$. As mentioned in the definition of a Hermitian metric in Section \ref{sec_def}, we shall not distinguish between the associated (1,1)-form of a Hermitian metric and the Hermitian metric itself. We want to show that the Hermitian metric on $X$ defined by $\Psi_\lambda = \Phi + \lambda \pi^*\omega_Y$ has positive holomorphic sectional curvature on $X$ if $\lambda$ is chosen large enough.
\par Let $p$ be a point in $X$ which lies in some fiber $X_0$. Since $\pi$ is of maximal rank everywhere, locally there is a neighborhood $U$ of $p$ such that $U = W \times V$, where $V$ is a neighborhood of $\pi(p)$ in $Y$, and $W$ is a neighborhood of $p$ in the fiber $X_0$. We may assume $V$ and $W$ are coordinate neighborhoods with local coordinates $(z_{s+1},\ldots,z_n)$ in $V$ and $(z_1,\ldots,z_s)$ in $W$. Then, $(z_1,\ldots,z_n)$ is a coordinate system around $p$ in $U$. For computational purposes, we will choose $(z_{s+1},\ldots,z_n)$ such that $\frac{\partial}{\partial z_{s+1}}, \ldots, \frac{\partial}{\partial z_n}$ are orthonormal at $\pi(p)$ with respect to the metric $\omega_Y$. Let
\[\widetilde{\Phi} = \frac{\sqrt{-1}}{2} \sum_{i,j=1}^n g_{i \bar j}(z_1,\ldots,z_n) \, dz_i \wedge d\bar{z}_j,\]
and
\[\omega_Y = \frac{\sqrt{-1}}{2} \sum_{\alpha, \beta = s+1}^n \widetilde{g}_{\alpha \bar \beta} (z_{s+1}, \ldots, z_n) \, dz_{\alpha} \wedge d\bar{z}_{\beta}, \qquad \text{with} \; \widetilde{g}_{\alpha \bar \beta}(p) = \delta_{\alpha \bar \beta}.\]
Consequently, the Hermitian metric $\Psi_\lambda$, after replacing w.l.o.g. $\lambda + \mu_0$ with $\lambda$, is given by
\begin{equation*}
\begin{aligned}
  & \sum_{i,j=1}^n h_{i\bar j} \, dz_i \otimes d\bar z_j \\
     =\ & \bigg(\sum_{i,j=1}^s g_{i\bar  j} \, dz_i \otimes d\bar z_j + \sum_{i=1}^s \sum_{\beta=s+1}^n g_{i\bar\beta} \, dz_i \otimes d\bar z_\beta + \sum_{\alpha=s+1}^n \sum_{j=1}^s g_{\alpha\bar j} \, dz_\alpha \otimes d\bar z_j \\
       & + \sum_{\alpha,\beta = s+1}^n \big(g_{\alpha\bar\beta} + \lambda\widetilde g_{\alpha\bar\beta}\big) \, dz_\alpha\otimes d\bar z_\beta\bigg).
\end{aligned}
\end{equation*}

\par Let $A$ be the $s \times s$ matrix with coefficients $g_{a\bar b}(p)$ for $1 \leq a, \, b \leq s$, and $A_{ab}$ be the $(a,b)^{th}$ cofactor of the matrix $A$. Then, using the following formula to compute the determinant of a block matrix:
\[\det\begin{pmatrix}
   P & Q \\
   R & S
\end{pmatrix} = \det(P)\,\det(S-RP^{-1}Q),\]
we obtain the following expressions for the coefficients of the inverse matrix, with the assumption that $1 \leq a,b \leq s$, and $s+1 \leq \chi,\eta \leq n$,
\begin{equation}\label{coefficients_inverse}
\begin{aligned}
   h^{a\bar b}(p) =\ & \frac{\lambda^{n-s}\det A_{ab} + O(\lambda^{n-s-1})}{\lambda^{n-s}\det A + O(\lambda^{n-s-1})}, \\
   h^{\chi\bar\chi}(p) =\ & \frac{\lambda^{n-s-1}\det A + O(\lambda^{n-s-2})}{\lambda^{n-s}\det A + O(\lambda^{n-s-1})}, \\
   h^{a\bar\eta}(p) =\ h^{\chi\bar b}(p) =\ & O(\lambda^{-1}), \qquad \text{and} \qquad h^{\chi\bar\eta}(p) =\ O(\lambda^{-2}), \qquad \chi \neq \eta.
\end{aligned}
\end{equation}\par 

Now, we check the conditions of Lemma \ref{lem1}:
\begin{enumerate}
  \item For $1 \leq i,j,k,l \leq s$, the formula \eqref{curv_formula} implies that the components of the curvature tensor are given by
          \begin{equation*}
           \begin{aligned}
              R_{i\bar jk\bar l}(p) =\ & -\frac{\partial^2 g_{i\bar j}}{\partial z_k \partial\bar z_l}(p) + \sum_{a,b=1}^s h^{a\bar b}(p) \frac{\partial g_{i\bar a}}{\partial z_k}(p) \frac{\partial g_{b\bar j}}{\partial\bar z_l}(p) \\
                 & + \sum_{\neg(u,v \leq s)} h^{u\bar v}(p) \frac{\partial g_{i\bar u}}{\partial z_k}(p) \frac{\partial g_{v\bar j}}{\partial \bar z_l}(p).
           \end{aligned}
           \end{equation*}
          The dependence of the coefficients of the inverse matrix on $\lambda$, as given in \eqref{coefficients_inverse}, reduces the above expression to
         \begin{equation*}
            \begin{aligned}
              R_{i\bar jk\bar l}(p) = & -\frac{\partial^2 g_{i\bar j}}{\partial z_k \partial\bar z_l}(p) + \sum_{a,b=1}^s h^{a\bar b}(p) \frac{\partial g_{i\bar a}}{\partial z_k}(p) \frac{\partial g_{b\bar j}}{\partial\bar z_l}(p) + O(\lambda^{-1}).
            \end{aligned}
         \end{equation*}\par
         We note that $\lim_{\lambda \rightarrow \infty} h^{a\bar b}(p) = \frac{\det A_{ab}}{\det A}$ and therefore          
         \begin{equation}\label{lem1_cond1a}
           \begin{aligned}
                      & \lim_{\lambda \rightarrow \infty} R_{i\bar jk\bar l}(p) \\
                  = & \lim_{\lambda \rightarrow \infty}\Bigg(-\frac{\partial^2 g_{i\bar j}}{\partial z_k \partial\bar z_l}(p) + \sum_{a,b=1}^s h^{a\bar b}(p) \frac{\partial g_{i\bar a}}{\partial z_k}(p) \frac{\partial g_{b\bar j}}{\partial\bar z_l}(p) + O(\lambda^{-1})\Bigg) \\
                  = & -\frac{\partial^2 g_{i\bar j}}{\partial z_k \partial\bar z_l}(p) + \sum_{a,b=1}^s \frac{\det A_{ab}}{\det A} \frac{\partial g_{i\bar a}}{\partial z_k}(p) \frac{\partial g_{b\bar j}}{\partial\bar z_l}(p).
           \end{aligned}
         \end{equation}
          The expression
             \begin{equation}\label{lem1_cond1b}
              2\sum_{i,j,k,l = 1}^s \Bigg(-\frac{\partial^2 g_{i\bar j}}{\partial z_k \partial\bar z_l}(p) + \sum_{a,b=1}^s \frac{\det A_{ab}}{\det A} \frac{\partial g_{i\bar a}}{\partial z_k}(p) \frac{\partial g_{b\bar j}}{\partial\bar z_l}(p)\Bigg) \, \xi_i\bar\xi_j\xi_k\bar\xi_l
           \end{equation}
is the numerator of the holomorphic sectional curvature of $\varphi_t$ on the fiber $\pi^{-1}(t)$ which, by assumption, is positive if the vector $(\xi_1,\ldots,\xi_s)\not = 0$. Comparing \eqref{lem1_cond1a} and \eqref{lem1_cond1b}, we conclude that for a sufficiently large choice of $\lambda$, the expression
  \begin{equation*}
          \begin{aligned}
            & \sum_{i,j,k,l = 1}^s R_{i\bar jk\bar l} \xi_i\bar\xi_j\xi_k\bar\xi_l \\
             =\ & \sum_{i,j,k,l=1}^s \bigg(-\frac{\partial^2 g_{i\bar j}}{\partial z_k \partial\bar z_l}(p) + \sum_{a,b=1}^s h^{a\bar b}(p) \frac{\partial g_{i\bar a}}{\partial z_k}(p) \frac{\partial g_{b\bar j}}{\partial\bar z_l}(p) \\
               & + O(\lambda^{-1})\bigg) \xi_i\bar\xi_j\xi_k\bar\xi_l
          \end{aligned}
          \end{equation*}
is positive if the vector $(\xi_1,\ldots,\xi_s)\not = 0$. This proves that the first condition of Lemma \ref{lem1} is satisfied for $\lambda$ sufficiently large.
  
\item If $\min(i,j,k,l) \leq s$ and $\max(i,j,k,l) > s$, then the dependence of the coefficients of the inverse matrix on $\lambda$, as described in \eqref{coefficients_inverse}, and the fact that the $\widetilde{g}_{\alpha\bar\beta}$ are functions of $z_{s+1},\ldots,z_n$ only, imply that the formula \eqref{curv_formula} gives $|R_{i\bar jk\bar l}(p)| \leq O(1)$ when $\min(i,j,k,l) \leq s$ and $\max(i,j,k,l) > s$. Therefore, the second condition of Lemma \ref{lem1} is also satisfied when $\lambda$ is large enough.

\item $s+1 \leq i,j,k,l \leq n$: In this case, we consider the submanifold $M^\prime$ around $p$ defined by $\{z_1 = \ldots = z_s = 0\}$. Let $G^\prime$ be the induced metric of $\Phi$ on $M^\prime$, and $\widetilde G^\prime$ be the induced metric of $\pi^*\omega_Y$ on $M^\prime$ so that (replacing $\lambda + \mu_0$ by $\lambda$ again)
 \[G^\prime + \lambda\widetilde G^\prime = \sum_{\alpha,\beta = s+1}^n \big(g_{\alpha\bar\beta} + \lambda \widetilde g_{\alpha\bar\beta}\big) \, dz_\alpha \otimes d\bar z_\beta\]
is the induced metric of $\Psi_\lambda$ on $M^\prime$. Clearly, $\widetilde G^\prime$ has positive holomorphic sectional curvature at $p \in M^\prime$. Lemma \ref{lem2} implies that $G^\prime + \lambda\widetilde G^\prime$ also has positive holomorphic sectional curvature at $p$ for a sufficiently large choice of $\lambda$. Therefore, the numerator of \eqref{hol_sect_curv_def} is positive with respect to $G^\prime + \lambda\widetilde G^\prime$, i.e., if $R^\prime_{\alpha\bar\beta\gamma\bar\delta}$ denote the components of the curvature tensor (for $\alpha,\beta,\gamma,\delta = s+1,\dots,n$) with respect to the induced metric $G^\prime + \lambda\widetilde G^\prime$, then
\[\sum_{\alpha,\beta,\gamma,\delta = s+1}^n R^\prime_{\alpha\bar\beta\gamma\bar\delta}(p) \xi_\alpha\bar\xi_\beta\xi_\gamma\bar\xi_\delta > 0\]
for $(\xi_{s+1},\ldots,\xi_n)\not = 0$. \par
The decreasing property of holomorphic sectional curvature on submanifolds implies that
   \begin{equation*}
        \begin{aligned}
          K\Big(\Psi_\lambda, (0,\ldots,0,\xi_{s+1},\ldots,\xi_n)\Big)(p) \geq\ & K\Big(\Psi_\lambda|_{M^\prime}, (\xi_{s+1},\ldots,\xi_n)\Big)(p) \\
             =\ & K\big(G^\prime + \lambda\widetilde G^\prime, (\xi_{s+1},\ldots,\xi_n)\big)(p),
        \end{aligned}
        \end{equation*}
     i.e.,
     \begin{equation*}
        \begin{aligned}
               & \frac{2\sum_{\alpha,\beta,\gamma,\delta = s+1}^n R_{\alpha\bar\beta\gamma\bar\delta}(p) \, \xi_\alpha\bar\xi_\beta\xi_\gamma\bar\xi_\delta}{\sum_{\alpha,\beta,\gamma,\delta = s+1}^n \big(g_{\alpha\bar\beta}(p) + \lambda\widetilde g_{\alpha\bar\beta}(p)\big)\big(g_{\gamma\bar\delta}(p) + \lambda\widetilde g_{\gamma\bar\delta}(p)\big) \, \xi_\alpha\bar\xi_\beta\xi_\gamma\bar\xi_\delta} \\
          \geq\ & \frac{2\sum_{\alpha,\beta,\gamma,\delta = s+1}^n R^\prime_{\alpha\bar\beta\gamma\bar\delta}(p) \, \xi_\alpha\bar\xi_\beta\xi_\gamma\bar\xi_\delta}{\sum_{\alpha,\beta,\gamma,\delta = s+1}^n \big(g_{\alpha\bar\beta}(p) + \lambda\widetilde g_{\alpha\bar\beta}(p)\big)\big(g_{\gamma\bar\delta}(p) + \lambda\widetilde g_{\gamma\bar\delta}(p)\big) \, \xi_\alpha\bar\xi_\beta\xi_\gamma\bar\xi_\delta},
        \end{aligned}
        \end{equation*}
  which implies that
    \begin{equation} \label{tensor_dec}
            \sum_{\alpha,\beta,\gamma,\delta = s+1}^n R_{\alpha\bar\beta\gamma\bar\delta}(p) \, \xi_\alpha\bar\xi_\beta\xi_\gamma\bar\xi_\delta \geq \sum_{\alpha,\beta,\gamma,\delta = s+1}^n R^\prime_{\alpha\bar\beta\gamma\bar\delta}(p) \, \xi_\alpha\bar\xi_\beta\xi_\gamma\bar\xi_\delta.
         \end{equation}
The expression on the right hand side of the above inequality is positive for $\lambda$ sufficiently large and for $(\xi_{s+1},\ldots,\xi_n)\not = 0$. Therefore, the expression on the left hand side is also positive. This proves the third and last condition of Lemma \ref{lem1} for $\lambda$ sufficiently large.
\end{enumerate}\par 
According to \eqref{key_remark} in Remark \ref{Tensor_sum_order}, the right hand side of the inequality \eqref{tensor_dec}, and consequently the left hand side of the inequality, in fact satisfy a stronger positivity statement in terms of $\lambda$, namely
\[\sum_{\alpha,\beta,\gamma,\delta = s+1}^n R_{\alpha\bar\beta\gamma\bar\delta}(p) \, \xi_\alpha\bar\xi_\beta\xi_\gamma\bar\xi_\delta \geq\ O(\lambda) \sum_{\alpha, \beta = s+1}^n \xi_\alpha\bar\xi_\alpha\xi_\beta\bar\xi_\beta.\]
Therefore, the inequality $K_2/K_1 \geq \mathcal{K}$ in the statement of Lemma \ref{lem1} is satisfied for $\lambda$ large enough, and hence, the lemma implies that $\Psi_\lambda$ has positive holomorphic sectional curvature at $p$ for $\lambda$ sufficiently large.
\par Let $U_{p,\lambda}$ be a neighborhood of $p$ in which the holomorphic sectional curvature with respect to $\Psi_\lambda$ is positive everywhere, i.e., the holomorphic sectional curvature at every point $q \in U_{p,\lambda}$ is positive in all the directions. Then, \cite[Lemma 4]{Wu} implies that the holomorphic sectional curvature at $q$ in a direction $t \in T_qX$ (with respect to $\Psi_\lambda$) is equal to the Gaussian curvature at $q$ with respect to the induced metric of $\Psi_\lambda$ on a 1-dimensional submanifold tangent to $t$. Dependence of the Gaussian curvature on $\lambda$ as in \eqref{HSC_curv_dec} implies that if $K(\Psi_\lambda, t)(q) > 0$, then $K(\Psi_{\lambda^\prime}, t)(q) > 0$ for any $\lambda^\prime > \lambda$.
\par Consequently, using the compactness property of $X$, we conclude that there exists a sufficiently large value of $\lambda$ such that the holomorphic sectional curvature of $X$ with respect  to $\Psi_\lambda$ is positive everywhere.\qed

\section{Remarks on the case when $\pi$ is not everywhere a submersion}\label{singularity}
Suppose the holomorphic map $\pi:\ X \rightarrow Y$ in Theorem \ref{mthm} is not everywhere a submersion, i.e., there exists at least one singular fiber. We would like to briefly explain why our method in general fails in this case. Note that for any positive number $\mu$, the Hermitian $(1,1)$-form $\widetilde \Phi + \mu\pi^*\omega_Y$ is positive definite at those points where $\pi$ is a submersion and semi-positive definite where it is not a submersion. Let us denote by $X_\text{sub}$ the set of all points of $X$ where $\pi$ is a submersion. For every $p \in X_\text{sub}$, we proved the existence of a (large) value of $\lambda$ such that there exists a neighborhood $U_{p,\lambda}$ of $p$ in which the holomorphic sectional curvature with respect to $\Psi_\lambda= \Phi + \lambda \pi^*\omega_Y$ is positive. We proved the above statement using the inequality \eqref{key_remark} to apply Lemma \ref{lem1}. Now, the key point is that the expression \eqref{lambda_exp} shows that as $p$ goes to the boundary of $X_\text{sub}$ and therefore $\det (\tilde g_{\alpha\bar\beta}(p))$ goes to zero, the value of $\lambda$ may have to be chosen larger and larger, with no finite limit existing,  in order to obtain the inequality \eqref{key_remark}. Thus, a uniform form $\Psi_\lambda$ cannot necessarily be chosen on all of $X_\text{sub}$. We can only do so on a compact subset of $X_\text{sub}$.\par

In other words, it is crucial for our method that the Rank Theorem can be applied to $\pi$ everywhere on $X$. To illustrate this point, we give two local examples in which $\pi$ is not a submersion and our method fails. In the first example, $\pi$ has a multiple fiber of multiplicity $k\geq 2$. In the second, $\pi$ has a singular fiber with a nodal singularity.

\begin{example}\label{multiple_fiber}
Let $X = D_1 \times D_2$ be the bi-disk, $Y = D_2$, and $(z_1,z_2)$ be the coordinate system in $X$. Assume that the holomorphic map $\pi: X \rightarrow Y$ is given by $(z_1,z_2) \mapsto z_2^k$, for some fixed integer $k\geq 2$, which is surjective and a submersion at all points except for the central fiber $D_1 \times \{0\}$. On the fibers, we consider the restriction of the following tensor on $X$ (which is a constant multiple of the Fubini-Study metric along each fiber) varying smoothly with respect to the base points:
\begin{equation*}
\Phi =\ \frac{1+z_2\bar z_2}{(1+z_1\bar z_1)^2}\, dz_1 \otimes d\bar z_1.
\end{equation*}
On the base $Y$, we consider the Fubini-Study metric given by 
\begin{equation*}
\omega_Y =\ \frac{1}{(1+z_2\bar z_2)^2}\, dz_2 \otimes d\bar z_2.
\end{equation*}
The warped product metric on $X\setminus\{z_2=0\}$ used in our proof then is the following:
\begin{equation*}
\begin{aligned}
 \Psi_\lambda  =&\  \Phi + \lambda\pi^*\omega_Y \\
    = &\ \frac{1+z_2\bar z_2}{(1+z_1\bar z_1)^2}\, dz_1 \otimes d\bar z_1 +\lambda \frac{k^2z_2^{k-1}\bar z_2^{k-1}}{(1+z_2^k\bar z_2^k)^2}\, dz_2 \otimes d\bar z_2.
\end{aligned}
\end{equation*}	

An explicit computation yields that the only non-zero coefficients of the curvature tensor are $R_{1\bar11\bar1}, R_{1\bar12\bar 2},$ and $R_{2\bar22\bar2}$, and the holomorphic sectional curvature of $X$ at a point $(z_1,z_2)$ ($z_2\not = 0$) with respect to $\Psi_\lambda$ in the direction of a unit tangent vector $\xi = \xi_1\frac{\partial}{\partial z_1} + \xi_2\frac{\partial}{\partial z_2}$ is given by
\begin{equation}\label{hsc_f-s}
\begin{aligned}
   K(\xi) =\ & 2\sum_{i,j,k,l=1}^2 R_{i\bar jk\bar l}(z_1,z_2) \xi_i\bar\xi_j\xi_k\bar\xi_l \\
          =\ & 2\big(R_{1\bar11\bar1}\xi_1\bar\xi_1\xi_1\bar\xi_1 + R_{1\bar12\bar2}\xi_1\bar\xi_1\xi_2\bar\xi_2 + R_{2\bar22\bar2}\xi_2\bar\xi_2\xi_2\bar\xi_2\big) \\
          =\ & 2\Bigg(\frac{2(1+z_2\bar z_2)}{(1+z_1\bar z_1)^4} \xi_1\bar\xi_1\xi_1\bar\xi_1 - \frac{1}{(1+z_1\bar z_1)^2(1+z_2\bar z_2)}\xi_1\bar\xi_1\xi_2\bar\xi_2 \\
             & + \lambda \frac{2k^4z_2^{2k-2}\bar z_2^{2k-2}}{\big(1+z_2^k\bar z_2^k\big)^4}\xi_2\bar{\xi}_2\xi_2\bar{\xi}_2\Bigg).
\end{aligned}
\end{equation}

The condition of $\xi = \xi_1\frac{\partial}{\partial z_1} + \xi_2\frac{\partial}{\partial z_2}$ being a unit tangent vector with respect to $\Psi_\lambda$ translates to the following equation:
\[\frac{1+z_2\bar z_2}{(1+z_1\bar z_1)^2}\xi_1\bar\xi_1 + \lambda\frac{k^2z_2^{k-1}\bar z_2^{k-1}}{(1+z_2^k\bar z_2^k)^2}\xi_2\bar\xi_2 = 1\]
\begin{equation}\label{metric-quad}
\Leftrightarrow\;\xi_2\bar\xi_2 = \frac{z_2^{-k+1}\bar z_2^{-k+1}(1+z_2^k\bar z_2^k)^2\Big(1 - \frac{(1+z_2\bar z_2)\xi_1\bar\xi_1}{(1+z_1\bar z_1)^2}\Big)}{\lambda k^2}.
\end{equation}
Substituting \eqref{metric-quad} in \eqref{hsc_f-s}, we obtain
\begin{equation}\label{hsc_f-s-1}
\begin{aligned}
     K(\xi) =\ & \frac{4(1+z_2\bar z_2)\xi_1\bar\xi_1\xi_1\bar\xi_1}{(1+z_1\bar z_1)^4} + \frac{4\Big(-1+\frac{(1+z_2\bar z_2)\xi_1\bar\xi_1}{(1+z_1\bar z_1)^2}\Big)^2}{\lambda} \\
            & + \frac{2z_2^{-k+1}\bar z_2^{-k+1}(1+z_2^k\bar z_2^k)^2\Big(-1+\frac{(1+z_2\bar z_2)\xi_1\bar\xi_1}{(1+z_1\bar z_1)^2}\Big)\xi_1\bar\xi_1}{\lambda k^2(1+z_1\bar z_1)^2(1+z_2\bar z_2)}.
\end{aligned}
\end{equation}

For $k\geq 2$, in case $\xi_1\not = 0$, the factor $z_2^{-k+1}\bar z_2^{-k+1}$ in the third term of the equation \eqref{hsc_f-s-1} causes $K(\xi)$ to diverge when $z_2 \rightarrow 0$. Therefore, the holomorphic sectional curvature is not defined on the singular fiber $D_1 \times \{0\}$. Moreover, because of the possibly negative coefficient of $z_2^{-k+1}\bar z_2^{-k+1}$, $K(\xi)$ is negative for small absolute values of $z_2$ unless the negative contribution is diminished by some sufficiently large value of $\lambda$ (depending on $z_2$) mitigating the multiplication by $z_2^{-k+1}\bar z_2^{-k+1}$. For example, for $k = 10$, i.e., for the map $(z_1,z_2) \mapsto z_2^{10},\, z_1 = 0.5,\, z_2 = 0.001$, and $\xi_1$ such that $\xi_1\bar \xi_1 = 0.01$, we obtain
\[\xi_2\bar\xi_2 \approx \frac{9.936 \times 10^{51}}{\lambda},\]
and therefore
\[K(\xi) \approx 0.00016384 - \frac{1.27181 \times 10^{50}}{\lambda}.\]
Clearly, as we consider points closer and closer to the singular fiber $D_1 \times \{0\}$, i.e., for smaller and smaller values of $z_2$,  the value of $\lambda$ is needed to be chosen larger and larger, with no finite limit existing, in order to obtain positive holomorphic sectional curvature at these points.\qed
\end{example}
\begin{example}
Let $X = \{(z_1,z_2)\in\mathbb{C}:|z_1|, |z_2|<2\}$ with local coordinates $(z_1,z_2)$ and $Y = \{z\in\mathbb{C}:|z|<4\}$ with local coordinate $z$. Consider the holomorphic map $\pi: X \rightarrow Y$ given by $(z_1,z_2) \mapsto z_1z_2$. The map $\pi$ is surjective and a submersion at all points except at the origin. On the fibers, we consider the restriction of the following tensor on $X\setminus\{(0,0)\}$ varying smoothly with respect to the base points:
\begin{equation*}
\Phi = \frac{\big(e^{-z_2\bar z_2}+e^{-z_1\bar z_1}\big)e^{z_1\bar z_1z_2\bar z_2}}{(z_1\bar z_1+z_2\bar z_2)^2}\begin{pmatrix}
                                    z_1\bar z_1 & -z_2\bar z_1 \\
                                   -z_1\bar z_2 & z_2\bar z_2
                                 \end{pmatrix}.
\end{equation*}
We obtained $\Phi$ (after some experimentation) by considering a certain family of Hermitian metrics with positive holomorphic sectional curvature on the fibers of $\pi$ and using orthogonal projections with respect to the standard Hermitian metric on $X$, as it was done at the beginning of Section \ref{proof_mthm}.\par
On the base $Y$, we consider the metric given by 
\begin{equation*}
\omega_Y =\ (10+z+\bar z)\, dz \otimes d\bar z,
\end{equation*}
which has positive holomorphic sectional curvature. The warped product metric on $X\setminus\{(0,0)\}$ constructed in our method then is the following:
\begin{equation*}\label{warp-metric-51}
\begin{aligned}
  &  \Psi_\lambda = \Phi + \lambda\pi^*\omega_Y =\\
   & \frac{\big(e^{-z_2\bar z_2}+e^{-z_1\bar z_1}\big)e^{z_1\bar z_1z_2\bar z_2}}{(z_1\bar z_1+z_2\bar z_2)^2}\begin{pmatrix}
                                    z_1\bar z_1 & -z_2\bar z_1 \\
                                   -z_1\bar z_2 & z_2\bar z_2
                                 \end{pmatrix} + \lambda (10+z_1z_2+\bar z_1\bar z_2)\begin{pmatrix}
                                         z_2\bar z_2 & z_1\bar z_2 \\
                                         z_2\bar z_1 & z_1\bar z_1
                                    \end{pmatrix}.
\end{aligned}    
\end{equation*}\par
We used the software Mathematica to (numerically) determine the holomorphic sectional curvature $K$ of $\Psi_\lambda$ at various points and directions on $X$. Due to space constraints, we just illustrate our findings here by listing below the holomorphic sectional curvature at the point $z_1 = 1.5,\ z_2=c/1.5$ (lying on the fiber over $c$) in the tangent direction $\frac{\partial}{\partial z_1}$ for various values of $\lambda$ and $c$. 
\begin{enumerate}\setlength{\itemsep}{.5cm}
    \item At the point $(1.5,-10^{-2}/1.5)$:
         \begin{enumerate}
              \item $\lambda=1\ \Rightarrow K\approx-0.395568$.
              \item $\lambda=10\ \Rightarrow K\approx-0.339601$.
              \item $\lambda=10^2\ \Rightarrow K\approx0.145204$.
          \end{enumerate}
    \item  At the point $(1.5,-10^{-3}/1.5)$:
          \begin{enumerate}
              \item $\lambda=1\ \Rightarrow K\approx-0.401863$.
              \item $\lambda=10\ \Rightarrow K\approx-0.401292$.
              \item $\lambda=10^2\ \Rightarrow K\approx-0.395597$.
              \item $\lambda=10^3\ \Rightarrow K\approx-0.339497$.
              \item $\lambda=10^4\ \Rightarrow K\approx0.146328$.
          \end{enumerate}
    \item At the point $(1.5,-10^{-4}/1.5)$:          
    \begin{enumerate}
              \item $\lambda=10\ \Rightarrow K\approx-0.40192$.
              \item $\lambda=10^2\ \Rightarrow K\approx-0.401863$.
              \item $\lambda=10^3\ \Rightarrow K\approx-0.401292$.
              \item $\lambda=10^5\ \Rightarrow K\approx-0.339483$.
              \item $\lambda=10^6\ \Rightarrow K\approx0.146438$.
          \end{enumerate}
    \item At the point $(1.5,-10^{-5}/1.5)$:      
          \begin{enumerate}
              \item $\lambda=10^2\ \Rightarrow K\approx-0.401926$.
              \item $\lambda=10^3\ \Rightarrow K\approx-0.40192$.
              \item $\lambda=10^5\ \Rightarrow K\approx-0.401292$.
              \item $\lambda=10^7\ \Rightarrow K\approx-0.339482$.
              \item $\lambda=10^8\ \Rightarrow K\approx0.146449$.
          \end{enumerate}
    \item At the point $(1.5,-10^{-10}/1.5)$:   
          \begin{enumerate}
              \item $\lambda=10^5\ \Rightarrow K\approx-0.401926$.
              \item $\lambda=10^{10}\ \Rightarrow K\approx-0.401926$.
              \item $\lambda=10^{15}\ \Rightarrow K\approx-0.401292$.
              \item $\lambda=10^{17}\ \Rightarrow K\approx-0.339482$.
              \item $\lambda=10^{18}\ \Rightarrow K\approx0.14645$.
          \end{enumerate}
 \end{enumerate}\par
Again, the pattern is clear: as we consider smaller and smaller values of $c$ in the base, the value of $\lambda$ necessary to make $K$ positive over $c$ becomes larger and larger, with no finite limit existing. \qed
\end{example}
We conclude with the following comment.
\begin{remark}
We do not know if it is possible to obtain a result similar to our Theorem \ref{mthm} in the presence of singular fibers, i.e., in case $\pi$ is not a submersion, either by modifying the warping construction or by some other method. Such a result would seem natural to us and worthy of pursuit, although a proof probably would have to be more involved than the warping method we used in this paper.
\end{remark}

\end{document}